\documentclass{amsart}[10pt]
\usepackage{amsmath}
\usepackage{amssymb} 
\usepackage{amsthm} 

 \newtheorem{thm}{Theorem}
 \newtheorem{cor}{Corollary}
 \newtheorem{prop}{Proposition}
 \newtheorem{lem}{Lemma}
 
 {\theoremstyle{definition}
 \newtheorem{exam}{Example}}
 {\theoremstyle{definition}
 \newtheorem{rem}{Remark}}
  {\theoremstyle{definition}
 }
  
 \newcommand{\mF}{\mathcal{F}}
 \newcommand{\mA}{\mathcal{A}}
 \newcommand{\mR}{ \mathcal{R}}
 \newcommand{\mB}{ \mathcal{B}}
 \newcommand{\mN}{ \mathcal{N}}
  \newcommand{\mI}{ \mathcal{I}}
  \newcommand{\degree}{ \textrm{deg}\,}

 \newcommand{\C}{ \mathbb{C}} 
 \newcommand{\Z}{ \mathbb{Z}}
 \newcommand{\R}{ \mathbb{R}}
 \newcommand{\Q}{ \mathbb{Q}}  
     
 \newcommand{\dl}{ \langle \langle}
 \newcommand{\dr}{\rangle \rangle}

\def\diaCrossP{\unitlength.1em
  \begin{minipage}{15\unitlength}
    \begin{picture}(15,15)
      \put(0,0){\vector(1,1){15}}
      \qbezier(15,0)(15,0)(10,5)
      \qbezier(5,10)(0,15)(0,15)
      \put(0,15){\vector(-1,1){0}}
    \end{picture}
  \end{minipage}
}
\def\diaCrossN{\unitlength.1em
  \begin{minipage}{15\unitlength}
    \begin{picture}(15,15)
      \put(15,0){\vector(-1,1){15}}
      \qbezier(0,0)(0,0)(5,5)
      \qbezier(10,10)(15,15)(15,15)
      \put(15,15){\vector(1,1){0}}
    \end{picture}
  \end{minipage}
}

\def\diaProjected{\unitlength.1em
  \begin{minipage}{14\unitlength}
    \begin{picture}(14,14)
      \put(0,0){\vector(1,1){14}}
      \put(14,0){\vector(-1,1){14}}
    \end{picture}
  \end{minipage}
}

\begin{document}  
\title{A note on geometric constructions of bi-invariant orderings}
\author{Tetsuya Ito}
\address{Graduate School of Mathematical Science, University of Tokyo, Japan}
\email{tetitoh@ms.u-tokyo.ac.jp}
\urladdr{http://ms.u-tokyo.ac.jp/~tetitoh}
 
\begin{abstract}
   We construct bi-invariant total orderings of residually torsion-free nilpotent groups by using Chen's iterated integrals. This construction can be seen as a generalization of the Magnus ordering of the free groups, and equivalent to the classical construction which uses an iteration of central extensions. Our geometric construction provides a connection between bi-orderings and the rational homotopy theory.
\end{abstract}

 \maketitle

\section{Introduction}

 A total ordering $<$ of a group $G$ is called {\it left-invariant} (resp. {\it right-invariant}) if the relation $<$ is preserved by multiplication of $G$ from left (resp. right). That is, $a < b$ implies $ca <cb$ (resp. $ac < bc$) for all $a,b,c \in G$. If the ordering is both left- and right- invariant, the ordering is called {\it bi-ordering}. The group $G$ is said to be {\it bi-orderable} if $G$ has a bi-ordering.

The most simple way to define a total ordering is to use the lexicographical ordering.
Let $\{v_{i}: G \rightarrow \R\}_{i \in \mI}$ be a family of maps, not necessarily a homomorphism, indexed by a well-ordered set $\mI$.
 We say $\{v_{i}\}_{i \in \mI}$ is a {\it lexicographical expression} of a total ordering $<$ of $G$ if $a<b$ is equivalent to the sequence of reals $\{v_{i}(b)\}_{i \in \mI}$ is bigger than the sequence of reals $\{v_{i}(a)\}_{i \in \mI}$ with respect to the lexicographical ordering of $\R^{\mI}$.

One of the typical but non-trivial examples of bi-orderings is the Magnus ordering of the free groups. Let $F_{n}$ be the rank $n$ free group generated by $\{x_{1},\ldots,x_{n}\}$ and $\mu: F_{n} \rightarrow \R \dl X_{1},\ldots, X_{n}\dr$ be the Magnus expansion which is an injective homomorphism defined by $\mu(x_{i})=1+X_{i}$ (See \cite{mks}). 
Here $\R \dl X_{1},\ldots, X_{n} \dr $ represents the algebra of non-commutative formal power series of variables $\{X_{1},\ldots, X_{n}\}$. 
Let $\mI$ be the set of monomials of $\R \dl X_{1},\ldots, X_{n} \dr$. 
We define the {\sf DegLex}-ordering $<_{\textsf{DegLex}}$ on $\mI$ as follows. For monomials $X_{I}=X_{i_{1}}\cdots X_{i_{k}}$ and $X_{J}=X_{j_{1}}\cdots X_{j_{l}}$ we define $X_{I} <_{\textsf{DegLex}} X_{J}$ if $k<l$, or $k=l$ and the sequence of integers $(j_{1},\ldots ,j_{k})$ is bigger than $(i_{1},\ldots,i_{k})$ with respect to the lexicographical ordering of $\R^{k}$. 
For a monomial $I \in \mI$, let $\overline{v_{I}}: \R \dl X_{1},\ldots, X_{n} \dr \rightarrow \R$ be a map defined by $\overline{v_{I}}(\Sigma_{J \in \mI} \;r_{J} \cdot J) = r_{I}$ and  let $v_{I} = \overline{v_{I}}\circ\mu:F_{n} \rightarrow \R$.
The {\it Magnus ordering} $<_{M}$ is a total ordering of $F_{n}$ defined by a lexicographical expression $\{v_{I}\}_{I\in \mI}$. It is known that the Magnus ordering is bi-ordering. See \cite{rw} for details.

A group $G$ is called {\it residually torsion-free nilpotent} if $G_{\infty} = \bigcap_{i\geq 1} G_{i} = \{1\}$, where $G_{k}$ is the $k$-th dimension subgroup of $G$ defined by $G_{k}=\{ g \in G \: | \: g-1 \in J^{k}\}$. Here $J$ denote the augmentation ideal of the group ring $\Z G$. The free group is a typical example of residually torsion-free nilpotent groups.

The aim of this paper is to provide a new geometric point of view for bi-orderings of residually torsion-free nilpotent groups. In section 3, we construct bi-orderings which we call {\it holonomy orderings} by using the universal holonomy map, which is derived from Chen's iterated integral theory. Chen's iterated integral theory will be reviewed in section 2.
Our constructions turn out to be equivalent to the classically known construction of bi-orderings which uses group extensions (Proposition \ref{prop:CisH}), so our construction provides a geometrical background of the classical construction.

Since Chen's iterated integral theory is closely related to the rational homotopy theory, holonomy ordering construction provides new relationships between bi-orderings and the rational homotopy theory. For example, we will show that finite type invariants of the pure braid group $P_{n}$ provides a lexicographical expression of a bi-ordering of $P_{n}$ (Proposition \ref{prop:finite} ). In section 4 we study properties of holonomy orderings, and provide a generalization of holonomy ordering construction. 
\\

\textbf{Acknowledgments.} 
The author wishes to express his gratitudes to Toshitake Kohno for many useful conversations. This research was supported by JSPS Research Fellowships for Young Scientists.

\section{Chen's iterated integral and holonomy representation}

In this section we briefly review Chen's iterated integral theory. For details, see \cite{c}.
We restrict our attention to smooth manifolds, although there is a generalization of Chen's iterated integral theories for simplicial complexes \cite{h}, and our methods can also be applied for such a general iterated integrals.

\subsection{Iterated integrals}

Let $M$ be a connected smooth manifold with a base point. We denote the de Rham DGA of $M$ by $A^{*}_{DR}(M)$ and the based loop space of $M$ by $\Omega M$. 
Let $\Delta_{q}$ be the standard $q$-dimensional simplex defined by
\[ \Delta_{q}=\{(t_{1},\ldots, t_{q}) \in \R^{q} \: |\: 0 \leq t_{1} \leq \cdots \leq t_{q} \leq 1 \}. \]
and  $ev: \Omega M \times \Delta_{q} \rightarrow M^{q}$ be the evaluation map 
defined by 
\[ ev(\gamma, (t_{1},\ldots ,t_{q}) ) = ( \gamma(t_{1}), \ldots,  \gamma(t_{q}) ) . \]

Formally, the iterated integral map $\int$ is defined as the composition
\[ \int : A_{DR}^{*}(M) ^{\otimes k} \stackrel{\times}{\rightarrow} A_{DR}^{*}(M^{k}) \stackrel{ev^{*}}{\rightarrow} A_{DR}^{*}(\Delta_{k} \times \Omega M) \stackrel{\int_{\Delta_{k}}}{\rightarrow} A_{DR}^{*}(\Omega M) \]
  where $\times$ is a cross-product and $\int_{\Delta_{k}}$ is an integration along fiber.

In this paper, we mainly use iterated integrals of $1$-forms. For $1$-forms we have more explicit description of iterated integrals.
Let $\omega_{1},\ldots,\omega_{m}$ be $1$-forms of $M$, and $\gamma : I \rightarrow M $ be a piecewise smooth path. Let us denote the pull-back of $\omega_{i}$ by $\gamma$ by $ \gamma^{*}\omega_{i}= \alpha_{i}(t) dt$. Then the iterated integral $\int_{\gamma}\omega_{1}\cdots \omega_{q}$ along $\gamma$ is explicitly written as the multiple integral 
\[ \int_{\gamma} \omega_{1}\cdots \omega_{q} = \int_{0 \leq t_{1} \leq \cdots \leq t_{q} \leq 1} \alpha_{1}(t_{1}) \alpha_{2}(t_{2}) \cdots \alpha_{q}(t_{q}) dt_{1}dt_{2}\cdots dt_{q}. \]

\subsection{Formal homology connection, Chen's holonomy map}

In this section we define Chen's universal holonomy map. See \cite{c} chapter III for details. 

From now on, we always assume that the real homology group of $M$ is finite dimensional. 
Let $\{ X_{1},\ldots, X_{m}\}$ be a homogeneous, ordered basis of $H_{*}(M; \R)$ such that $\{X_{1},\ldots,X_{k}\}$ forms a basis of $H_{1}(M;\R)$.

Let $J$ be the augmentation ideal of $TH$, the tensor algebra of $H$ and $\widehat{TH} = \varprojlim TH\slash J^{N}$ be the nilpotent completion of $TH$.
Then $\widehat{TH}$ is identified with the algebra of non-commutative formal power series $\R \dl X_{1},\ldots ,X_{m}\dr$.
Let $\widehat{J}$ be the nilpotent completion of the augmentation ideal $J$ of $TH$, which is an ideal of $\widehat{TH}$ which consists of formal power series with zero constant term.
We define a grading of $\widehat{TH}$ by the formula $\degree X_{1}\cdots X_{k} = p_{1}+\cdots +p_{k} -k$ where $p_{i}$ is the degree of $X_{i}$, and define the involution $\varepsilon: A_{DR}^{*}(M) \rightarrow A_{DR}^{*}(M)$ by $\varepsilon(\omega)= (-1)^{\degree \omega} \omega$.

 Let $\delta$ be a degree $1$ derivation of $\widehat{TH}$ and $\overline{\omega} =  \sum _{i_{1},\ldots, i_{p}} \omega_{i_{1}\cdots i_{p}} X_{i_{1}}\cdots X_{i_{p}}$ be an element of $A_{DR}^{*}(M)\otimes \widehat{TH}$.
We call a pair $(\overline{\omega},\delta)$ is a {\it formal homology connection} if the following three conditions hold.
\begin{enumerate}
\item $\delta X_{i} \in \widehat{J}^{2}$.
\item $\degree \omega_{i_{1}\cdots i_{p}} = \degree X_{i_{1}}\cdots X_{i_{p}}$.
\item $d\overline{\omega} + \delta \overline{\omega} = \varepsilon (\overline{\omega}) \wedge \overline{\omega}$.
\end{enumerate}

Such a pair $(\overline{\omega},\delta)$ always exists and we can choose $\overline{\omega}$ so that the coefficient $\omega_{i}$ represents the cohomology class which is the dual of $X_{i} \in H_{*}(M;\R)$.
In fact, formal homology connections are inductively constructed from the de Rham DGA $A_{DR}^{*}(M)$.
There are many choices of formal homology connections, but in some cases there is a canonical choice of a formal homology connection. For example, if $M$ is a compact Riemannian manifold, one can find a canonical formal homology connection by using the de Rham-Hodge decomposition of $A_{DR}^{*}(M)$.

Let $\mR$ be the degree $0$ part of $\widehat{TH}$. $\mR$ is identified with $\widehat{TH_{1}} = \R \dl X_{1},\ldots, X_{k} \dr$, the nilpotent completion of the tensor algebra of $H_{1}(M;\R)$.
Let $\omega$ be the degree $0$ part of the formal homology connection $\overline{\omega}$ and $\mN$ be the degree $0$ part of $\delta(\widehat{TH})$, which is an ideal of $\mR$. Let us denote the quotient algebra $\mR \slash \mN$ (resp. $\mR \slash (\mN + \widehat{J}^{k})$) by $R$ (resp. $R_{k}$). We remark that $R$ is nothing but the $0$-th homology group of the complex $(\widehat{TH}, \delta)$.
The (Chen's) {\it holonomy representation} is a homomorphism $\Theta : \pi_{1}(M) \rightarrow R$ defined by 
\[ \Theta([\gamma]) = 1 + \sum_{i=1}^{\infty} \int_{\gamma} \underbrace{\omega \omega \cdots \omega}_{i} \]
and the {\it order $k$ holonomy representation} is a homomorphism $\Theta_{k} : \pi_{1}(M) \rightarrow R_{k}$ defined by 
\[ \Theta_{k}([\gamma]) = 1 + \sum_{i=1}^{k} \int_{\gamma} \underbrace{\omega \omega \cdots \omega}_{i} \]
 Chen's results are summarized as follows.
\begin{thm}[Chen \cite{c}]
Let $\Theta$, $\Theta_{k}$ be as above.
Then,
\begin{enumerate}
\item $\textrm{Ker}\,\Theta = \pi_{1} (M)_{\infty}$ and $\textrm{Ker}\,\Theta_{k}=\pi_{1} (M)_{k}$.
\item $\Theta$ induces an isomorphism of Hopf algebra $\Theta: \widehat{\R\pi_{M}} \rightarrow R$, where $\widehat{\R \pi_{M}}$ is the nilpotent completion of the group ring $\R \pi_{1}(M)$.
\item $\Theta_{k}$ induces an isomorphism of Hopf algebra $\Theta_{k}: \R\pi_{1}(M) \slash J^{k} \rightarrow R_{k}$.
\end{enumerate}
\end{thm}

Thus, if $\pi_{1}(M)$ is residually torsion-free nilpotent, then the universal holonomy map is injective.

\begin{rem}
If $M$ is simply connected (more generally, if $M$ is a nilpotent space), a formal homology connection computes the homology group of the based loop space $\Omega M$. More precisely, there exists a natural isomorphism of Hopf algebra
\[ H_{*}(\Omega M) \cong H_{*}(\widehat{TH},\delta) .\]
\end{rem}
\begin{rem}
To define the holonomy map, we do not need the whole formal homology connection $(\overline{\omega},\delta)$. Only we need is the degree $0$ and $1$ parts. From the construction of formal homology connection, the degree $\leq 1$ part is computed by using $A_{DR}^{\leq 2}(M)$, the degree $\leq 2$ part of de Rham DGA. In particular, it is sufficient to assume that $H_{\leq 2}(M;\R)$ is finite dimensional. 
\end{rem}

We say $M$ has a {\it quadratic formal homology connection} if the ideal $\mN$ is quadratic. That is, the ideal $\mN$ is generated by elements of $H_{1}(M) \otimes H_{1}(M)$.
Let $\cup^{*}: H_{2}(M) \rightarrow H_{1}(M) \otimes H_{1}(M)$ be the dual of the cup product. It is known that the quadratic part of the ideal $\mN$ is generated by the image of $\cup^{*}$. Thus, if $M$ has a quadratic formal homology connection, then the ideal $\mN$ is determined by the cup products.

The condition that $M$ has a quadratic formal homology connection is equivalent to the condition that $M$ is {\it formal} in the sense of Sullivan's rational homotopy theory \cite{s}. That is, the cohomology algebra $H^{*}_{DR}(M)$ viewed as DGA having zero differential is quasi-isomorphic to the de Rham DGA $A^{*}_{DR}(M)$.

\section{Holonomy construction of bi-invariant orderings}

 In this section we give a construction of bi-invariant orderings by using Chen's holonomy map, and provide some examples.
  
\subsection{Definition of holonomy orderings}

Let $M$ be a connected smooth manifold and $G$ be its fundamental group.
We always assume that $H_{\leq 2}(M ; \R)$ is finite dimensional and use the notations in section 2.

Let $R = \mF^{0}R \supset \mF^{1}R \supset \cdots $ be the decreasing filtration of $R$ induced by the powers of the ideal $\widehat{J}$.
This filtration is multiplicative, that is, $\mF^{k}R \cdot \mF^{l}R \subset \mF^{k+l}R$ holds.
Let us choose a subspace $R^{i}$ of $R$ so that $\mF^{k}= \bigoplus_{i \geq k} R^{i}$ holds. Then this filtration defines the grading $R=\bigoplus_{i=0}^{\infty} R^{i}$.

 Let $\mB = \{v_{i}\}_{  i \in \mI }$ be an ordered, homogeneous basis of $R$ such that the degree is non-decreasing. That is, $\degree v_{i} \leq \degree v_{j}$ if $i<j$.
Let $\{v_{i}^{*}: R \rightarrow \R\}_{i \in \mI}$ be the dual basis of $R$.
We denote by $<_{\mB}$ the total ordering of the vector space $R$ defined by the lexicographical expression ${v_{i}^{*}}$. That is, for $a,b \in R$, we define $a <_{\mB} b$ if the sequence of reals $\{v_{i}^{*}(b)\}_{i \in \mI}$ are bigger than $\{v_{i}^{*}(a)\}_{i \in \mI}$ with respect to the lexicographical ordering of $\R^{\mI}$.

The {\it holonomy ordering} is a total ordering $<$ of $G \slash G_{\infty}$ defined by $a<b$ if $\Theta(a) <_{\mB} \Theta(b)$. In other words, the holonomy ordering is an ordering defined by the lexicographical expression $\{v_{i}^{*}\circ \Theta\}_{i \in \mI}$. The following theorem is the main result of this paper.

\begin{thm}
\label{thm:main1}
The holonomy ordering is a bi-invariant total ordering of $G \slash G_{\infty}$.
\end{thm}
\begin{proof}
First observe that from the definition of Chen's holonomy map $\Theta$, the image of $\Theta$ lies in $1+\mF^{1}(R)$.
Now assume that $a<b$ for $a,b \in G$. Then we can write their images as
\[ \Theta(a) = 1 + A_{< i } + A_{i} + A_{> i}, \; \Theta(b) = 1+A_{< i} + B_{i} + B_{>i} \]
where $A_{<i}$ (resp. $A_{i}$, $A_{>i }$) is degree $<i$ (resp. $i$, $>i$) part. Since $a<b$, $A_{i} <_{\mB} B_{i}$ hold.
We show the left-invariance of the holonomy ordering $<$. The proof of right-invariance is similar.
For $g \in G$, let us denote its image by $\Theta(g) = 1 + G_{< i} + G_{i} + G_{>i}$, where $G_{<i}$ (resp. $G_{i},G_{>i}$) represents the degree $<i$ (resp. $i$, $i>$) part.  
Since $R_{i} \cdot R_{j} \subset \mF^{i+j} R$, the degree $< i$ part of $\Theta(ga)$ is the degree $<i$ part of $\Theta(g) \cdot(1 + A_{<_{i}})$. Similarly, the degree $i$ part of $\Theta(ga)$ is given by $A_{i} + G_{i} + [(1+A_{< i})(1+ G_{< i})]_{i}$, where $[(1+A_{< i})(1+ G_{< i})]_{i}$ represents the degree $i$ part of $(1+A_{< i})(1+ G_{< i})$.

On the other hand, the degree $<i$ part of $\Theta(gb)$ is the degree $<i$ part of $\Theta(g) \cdot(1 + A<_{i})$, and the degree $i$ part of $\Theta(gb)$ is given by $B_{i} + G_{i} + [(1+A_{< i})(1+ G_{< i})]_{i}$.
Thus, $\Theta(gb)-\Theta(ga) = B_{i} -A_{i} + \textrm{ (higher parts)}$, so we conclude $ga < gb$.
\end{proof}

This provides a geometric proof of bi-orderability of residually torsion-free nilpotent groups.

\begin{cor}
If a group $G$ is residually torsion-free nilpotent, then $G$ is bi-orderable.
\end{cor}

The map $v_{i}^{*}\circ \Theta : G \rightarrow \R$ is written as the iterated integral 
$ [\gamma] \mapsto \int_{\gamma} \omega_{1}\cdots \omega_{k}$ of some $1$-forms $\omega_{1},\ldots,\omega_{k}$. Thus, the lexicographical expression of the holonomy ordering is given as the iterated integrals. 

\subsection{Comparison with classical construction of bi-orderings}

We compare the holonomy ordering construction and the classical construction of bi-orderings for residually torsion-free nilpotent groups which use an iteration of group extension. 

The classical construction is based on the following lemma, whose proof is routine (see \cite{lrr}).

\begin{lem}
\label{lem:extension}
Let $H,K$ be a bi-orderable group  and  $ 1 \rightarrow H \rightarrow G \stackrel{p}{\rightarrow} K \rightarrow 1$
be a group extension.
For a bi-ordering $<_{K}$ of $K$ and a bi-ordering $<_{H}$ of $H$, define an ordering $<_{G}$ of $G$ by $g <_{G} g'$ if $p(g) <_{K} p(g')$ or, $p(g)=p(g')$ and $1<_{G} g^{-1}g'$.
If $<_{H}$ is invariant under the action of $G$, that is,  for all $h,h' \in H$ $h <_{H} h'$ implies $ghg^{-1}<_{H} gh'g^{-1}$ for all $g \in G$, then $<_{G}$ is a bi-ordering.
\end{lem}

Now using Lemma \ref{lem:extension}, we construct a bi-ordering of a residually torsion-free nilpotent group $G$ as follows.
We inductively construct a bi-ordering $<_{k}$ of $G \slash G_{k}$ for each $k>0$. To begin with, $G \slash G_{1} =\{1\}$ so let $<_{1}$ be the trivial bi-ordering.

To define $<_{k+1}$, first we choose a bi-invariant ordering $<'_{k}$ of torsion-free abelian group $ G_{k} \slash G_{k+1}$. 
Observe that there is a central extension
\[ 1 \rightarrow G_{k}\slash G_{k+1} \rightarrow G \slash  G_{k+1} \rightarrow G \slash  G_{k} \rightarrow 1.\]
From this central extension and bi-orderings $<_{k}$, $<'_{k}$, we construct an ordering $<_{k+1}$ of $G \slash G_{k+1}$ as in Lemma \ref{lem:extension}. Since the extension is central, the assumption of Lemma \ref{lem:extension} is satisfied so $<_{k+1}$ is indeed a bi-ordering.

Since $G$ is torsion-free nilpotent, the sequence of bi-orderings $<_{k}$ defines a bi-ordering $<_{G}$ of $G$.
We call this bi-ordering $<_{G}$ an {\it iterated extension ordering} of $G$ derived from $\{<'_{k}\}$.

Now we show this iterated extension construction is equivalent to the holonomy construction given in section 3.1.

\begin{prop}
\label{prop:CisH}
Let $G= \pi_{1}(M)$ be a residually torsion-free nilpotent group.
Then each holonomy ordering is an iterated extension ordering and conversely, each iterated extension ordering is a holonomy ordering. Thus, the holonomy ordering construction is equivalent to the iterated extension ordering construction.
\end{prop}
\begin{proof}
We use notations in section 2 and 3.1. 
Let $\{v_{i}\}$ be an ordered, homogeneous, degree non-decreasing basis of $R$ which defines a holonomy ordering $<_{H}$. 
By definition of $R_{i}$, the universal holonomy map provides an isomorphism $( G_{k} \slash G_{k+1}) \otimes \R \cong R_{i}$. So we can regard $G_{k} \slash G_{k+1}$ as an integer lattice of $R_{i}$.
Let $<'_{k}$ be the restriction of the ordering $<_{H}$ to $G_{k} \slash G_{k+1}$.
Then $<'_{k}$ is a bi-ordering, and the holonomy ordering $<_{H}$ is nothing but the iterated extension ordering derived from this sequence of orderings $\{<'_{k}\}$.

Conversely, let $<_{G}$ be an iterated extension ordering of $G$ derived from $\{<'_{k}\}$.
Then the ordering $<'_{k}$ is naturally extended as a bi-ordering $\widetilde{<'_{k}}$ of $( G_{k} \slash G_{k+1}) \otimes \R = R_{i}$. 
Let us take an ordered basis $\{v_{(k), i}\}_{i=1,2,\ldots}$ of $R_{k}$ so that the dual basis 
$\{v_{(k), i}^{*}\}$ gives a lexicographical expression of the ordering $\widetilde{<'_{k}}$.
Now by correcting the basis $\{v_{(k),i }\}$ of $R_{k}$ for all $k$, we obtain an ordered, homogeneous, degree non-decreasing basis $\{v_{i}\}$ of $R$. The holonomy ordering defined by the basis $\{v_{i}\}$ is identical with $<_{G}$.
\end{proof}

Now we show that the holonomy ordering construction is indeed an extension of the Magnus ordering of the free group $F_{n}$. The following proposition is directly obtained from Proposition \ref{prop:CisH}, because the Magnus ordering is an iterated extension ordering. However, here we present a different and direct proof which is based on the fact that Magnus expansion is equivalent to Chen's holonomy map.

\begin{prop}
\label{prop:MisH}
The Magnus ordering $<_{M}$ of the free group is a holonomy ordering.
\end{prop}
\begin{proof}
Let $D_{n}$ be the $n$-punctured disc and $\{x_{1},\ldots,x_{n}\}$ be a generator of $\pi_{1}(D_{n}) = F_{n}$. Take $1$-forms $\omega_{1},\ldots,\omega_{n}$ of $D_{n}$ so that their representing cohomology classes $\{[\omega_{i}]\}$ are dual to $\{x_{i}\}$, and they satisfy the equation $\omega_{i} \wedge \omega_{j}=0 $. Then, the formal homology connection is taken as $(\omega = \sum_{i=1}^{n}\omega_{i} X_{i}, 0)$ and the holonomy representation defines an injective homomorphism 
\[ \Theta : F_{n} \rightarrow \R \dl X_{1},\ldots,X_{n}\dr.\]
Let us denote $\Theta(x_{i} ) = 1 + X_{i} + X^{\geq 2} _{i}$, where $X^{\geq 2}_{i}$ is the degree $\geq 2$ part.

 Let $\mB$ be $\textsf{DegLex}$-ordered monomial basis of $R = \dl X_{1},\ldots,X_{m}\dr$, and $<$ be the holonomy ordering defined by $(\omega,0)$ and $\mB$. 
 Let $\alpha$ be an automorphism of $\R \dl X_{1},\ldots,X_{n}\dr$ defined by $\alpha (X_{i} )= X_{i} + X^{\geq 2}_{i}$. Then $\Theta = \alpha \circ \mu$ holds. For each element $x \in \R \dl X_{1},\ldots,X_{n}\dr$, $\alpha$ preserves the lowest degree part of $x$. Hence if $a <_{M} b$ then $a < b$ holds, so these two orderings are identical.
\end{proof}

\subsection{Application: bi-ordering of pure braid groups and finite type invariants}

In this section we describe a relationship between holonomy orderings and finite type invariants of pure braids. 
This relation is based on the work of Kohno \cite{k1}, which relates iterated integral and finite type invariants of braids. For details of pure braid groups and their finite type invariants, see \cite{k1}.

For $1 \leq i < j \leq n$, let $H_{i,j}=\textrm{ker}\,(z_{i}-z_{j})$ be a hyperplane of $\C^{n}$. The hyperplane arrangement $\mA_{n} = \{ H_{i,j} \:|\: 1 \leq i < j \leq n\}$ is called the {\it braid arrangement}. We denote by $M_{\mA}$ the complement of the arrangement $\C^{n}- \bigcup_{H \in \mA} H$. The fundamental group of $M_{\mA}$ is the pure braid group $P_{n}$, and residually torsion-free nilpotent.

First of all, let us review the definition of finite type invariants.
A singular pure braid is a pure braid having transversal double points. We denote the set of singular pure braids having $k$ singular points by $S^{k}P_{n}$ and let $SP_{n}= \bigcup_{k \geq 0 }S^{k}P_{n}$.
Each map $v : P_{n}\rightarrow \R$ can be extended to the map $\widetilde{v}: SP_{n} \rightarrow \R$ by defining $\widetilde{v}(\diaProjected) =  v(\diaCrossP)-v(\diaCrossN)$. A map $v$ is called a {\it finite type invariant} of order $k$ if $\widetilde{v}(\beta)=0$ for all $\beta \in S^{\geq k}P_{n}$. 
Let $V_{k}(P_{n})$ be the set of order $k$ finite type invariants and $V(P_{n}) = \bigcup_{k >0} V_{k}(P_{n})$ be the set of all finite type invariants. Then $V_{k}(P_{n})$ and $V(P_{n})$ are $\R$-vector spaces.

Finite type invariants and holonomy orderings are related as follows.

 \begin{prop}
 \label{prop:finite}
 Let $\mB=\{v_{i}\}_{i \in \mI}$ be an ordered basis of $V(P_{n})$ such that the order of $v_{i}$ are non-decreasing.
 Let us define a total ordering $<_{\mB}$ of $P_{n}$ by $\alpha <_{\mB} \beta$ if $\{ v_{i}(\beta)\}_{i \in \mI}$ is bigger than $\{v_{i}(\alpha)\}_{i \in \mI}$ with respect to the lexicographical ordering. Then $<_{\mB}$ is a holonomy ordering, hence bi-invariant total ordering. Thus, the sequence of finite type invariants $\{v_{i}\}_{i \in \mI}$ gives a lexicographical expression of a holonomy ordering of $P_{n}$.
 \end{prop}
\begin{proof}
Let $\Theta: P_{n} \rightarrow R$ be a holonomy map.
There is a natural isomorphism $V_{k}(P_{n}) \cong \textrm{Hom}(\R P_{n} \slash J^{k+1} ,\R)$, so $R^{*}= \textrm{Hom}_{\R}(R, \R) \cong V(P_{n})$ \cite{k1}. Therefore a degree non-decreasing dual basis of the algebra $R$ corresponds to an order non-decreasing basis of $V(P_{n})$. 
Thus the basis $\mB$ provides a lexicographical expression of a holonomy ordering of $P_{n}$. 
\end{proof}

\begin{rem}
  The holonomy map $\Theta: P_{n} \rightarrow R$ is known as the universal finite type invariant of the pure braid groups (the universal representation of the quantum group representations of pure braid groups), which is a prototype of the Kontsevich invariant of knots. By using the Drinfel'd associator, we can construct the universal holonomy map over the rationals in more explicit form \cite{k2}. In particular, the conclusion of Proposition \ref{prop:finite} holds for $\Q$-valued finite type invariants.
\end{rem}

\section{Properties of holonomy orderings and generalization}

In this section we study structures and properties of the holonomy orderings, based on the relationship between Chen's theory and rational homotopy theory. 

\subsection{General case}

First of all, we study properties of holonomy orderings for general spaces.
The following proposition is a refinement of famous Stallings' Theorem \cite{st}.

\begin{prop}
\label{prop:structure1}
If $H_{2}(M ;\R)= 0$, then $\pi_{1}(M)$ contains the rank $b_{1}(M)$ free group $F$ such that the restriction of a holonomy ordering $<$ to $F$ is the Magnus ordering, where $b_{1}(M)$ is the 1st Betti number of $M$.
\end{prop}

\begin{proof} 
Let $\{X_{1},\ldots, X_{m}\}$ be a basis of $H_{1}(M;\R)$ and $\{g_{1},\ldots,g_{m}\}$ be elements of $\pi_{1}(M)$ such that $h(g_{i}) = X_{i}$, where $h :\pi_{1}(M) \rightarrow H_{1}(M;\R)$ is the Hurewicz homomorphism.
Let $\mB_{\sf DegLex}$ be an ordered basis of $R = \dl X_{1},\ldots,X_{m}\dr$ which consists of the monomials of $R$ and ordered by the {\sf DegLex-}ordering. 
Let us choose $1$-forms $\{\omega_{1},\ldots ,\omega_{m}\}$ on $M$ so that their representing cohomology classes are dual to $\{X_{1},\ldots, X_{m}\}$. Take a formal homology connection $\omega = \sum_{i=1}^{m} \omega_{i} + \cdots$. Since $H_{2}(M ; \R ) = 0$, the ideal $\mN$ must be trivial.
Let $\Theta: \pi_{1}(M) \rightarrow  \mR = \R \dl X_{1}, \ldots, X_{m}\dr$ be the holonomy map. 
We denote the holonomy ordering with respect to the basis $\mB_{{\sf DegLex}}$ by $<_{M}$. 
Let $F$ be a subgroup of $G$ generated by $\{g_{1},\ldots,g_{m}\}$. Then 
$\Theta(g_{i}) = 1 + X_{i} + \textrm{(Higher term)}$, so as in the proof of Proposition $\ref{prop:MisH}$, we conclude that $F$ is the rank $m$ free groups, and the restriction of $<_{M}$ to $F$ is the Magnus ordering.
\end{proof}

One of the benefits to consider the holonomy orderings is that we can associate the properties of groups with the properties of manifolds (topological space having the given group as its fundamental group). 
 
\begin{thm}
\label{thm:retract}
Let $N$ be a smooth manifold which is a retract of $M$ and $\iota : N \hookrightarrow M$ be the inclusion. Then for each holonomy ordering $<_{N}$ of $\pi_{1}(N)$, there exists a holonomy ordering $<_{M}$ of $\pi_{1}(M)$ such that the restriction of $<_{M}$ to $\iota_{*}(\pi_{1}(N))$ is $<_{N}$. Thus, every holonomy ordering of $N$ can be extended as a holonomy ordering of $M$.
\end{thm}

\begin{proof}
 Let $(\omega_{N},\delta_{N})$ be a formal homology connection of $N$ and $\mB_{N}$ be an ordered, degree non-decreasing basis of $R_{N}$ which defines the holonomy ordering $<_{N}$.
Let $r: M \rightarrow N$ be a retraction.
Since $N$ is a retract of $M$, the de Rham DGA of $M$ is written as a direct sum $A_{DR}^{*}(M)= r^{*}A_{DR}^{*}(N) \oplus A'$ for some DGA $A'$.
Thus, we can construct a formal homology connection $(\omega_{M},\delta_{M})$ so that it is an extension of $(\omega_{N}, \delta_{N})$. Hence, the non-commutative algebra $R_{M}$ is also written as a direct sum $R_{M} = R_{N} \oplus R'$ for some non-commutative algebra $R'$. 
Let $\mB_{M}$ be an ordered, degree non-decreasing basis of $R_{M}$ obtained by adding vectors to $\iota_{*}(\mB_{N})$. Then the holonomy ordering defined by $(\omega_{M},\delta_{M})$ and $\mB_{M}$ is an extension of the holonomy ordering of $<_{N}$.
\end{proof}

\subsection{Formal space case}

In this section we study properties of holonomy orderings for formal spaces.
In the case of formal space, the structure of holonomy ordering is determined by the cohomology algebra, in particular the cup product of the $1$st cohomology groups.

\begin{thm}
\label{thm:formal}
Assume that $M$ and $N$ are formal and that there exists a continuous map $\iota: N \rightarrow M$ which satisfies the following conditions.
\begin{enumerate}
\item $\iota_{*}:H_{\leq 2}(N;\R) \rightarrow H_{\leq 2}(M;\R)$ is injection. 
\item $\iota_{*}$ induces a decomposition $H_{\leq 2}(M;\R) = \iota_{*}H_{\leq 2}(N;\R) \oplus A$ which preserves the dual of the cup product.
That is, $\cup^{*}(\iota_{*}H_{2}(N;\R)) \subset \iota_{*}H_{1}(N;\R) \otimes \iota_{*}H_{1}(N;\R)$ and $\cup^{*}(A_{2}) \subset A_{1} \otimes A_{1}$ hold.
\end{enumerate}
Then $\iota_{*}: \pi_{1}(N) \rightarrow \pi_{1}(M)$ is also injection, and each holonomy ordering $<_{N}$ can be extended as a holonomy ordering $<_{M}$ of $\pi_{1}(M)$.
\end{thm}

\begin{proof}
First observe that $\iota$ induces an injection $ \widehat{T\iota} : \widehat{TH_{1}(N)} \rightarrow \widehat{TH_{1}(M)}$.
Moreover, since $M$ and $N$ are formal, $M$ and $N$ have quadratic formal homology connections.
Fix such a quadratic formal homology connection, and let $\Theta_{N}: \pi_{1}(N) \rightarrow \widehat{TH_{1}(N)} \slash \mN_{N} =R_{N}$ and $\Theta_{M}: \pi_{1}(M) \rightarrow \widehat{TH_{1}(N)} \slash \mN_{M} =R_{M}$ be holonomy maps.

Now the ideal $\mN_{N}$ and $\mN_{M}$ are generated by the image of the dual of the cup products.
Thus from the assumption we conclude that $\mN_{M}=\iota_{*}\mN_{N} + \mathfrak{a}$ holds, where $\mathfrak{a}$ is an ideal generated by $\cup^{*}(A_{2})$. Hence we conclude that $\widehat{T\iota}$ induces an injection $\widehat{T\iota_{*}} : R_{N} \hookrightarrow R_{M}$, thus $\iota_{*}:\pi_{1}(N) \rightarrow \pi_{1}(M)$ is also an injection.

Let $\mB_{N}$ be the ordered degree non-decreasing basis of $R_{N}$, which defines a holonomy ordering $<_{N}$ of $\pi_{1}(N)$. By adding vectors to $\iota_{*}\mB_{N}$, we form an ordered degree non-decreasing basis $\mB_{M}$ of $R_{M}$. Then the holonomy ordering $<_{M}$ defined by the basis $\mB_{M}$ provides an extension of the holonomy ordering $<_{N}$.
\end{proof}

\begin{exam}
Assume that $M$ is a formal space and the cup product $\cup: H^{1}(M;\R) \otimes H^{1}(M;\R) \rightarrow H^{2}(M;\R)$ is the zero map.
Then, by Theorem \ref{thm:formal} (take $N$ as the $b_{1}(M)$-punctured disc), we conclude that $\pi_{1}(M)$ contains the rank $b_{1}(M)$ free group $F$, and every holonomy ordering of $F$ can be extended as a holonomy ordering of $M$.
\end{exam}

\subsection{Generalizations of holonomy ordering and examples}

We close the paper by providing a generalized construction of the holonomy ordering construction which can be seen as a special case of group extension construction described in Lemma \ref{lem:extension}.

We call an ordered basis $\mB=\{v_{i}\}_{i \in \mI}$ of $R_{M} = \widehat{\R \pi_{1}M}$ is {\it bi-ordering basis} if $\{ v_{i}^{*}\circ\Theta_{M}\}_{i \in \mI}$ is a lexicographical expression of a bi-ordering $<$ of $\pi_{1}(M)\slash \pi_{1}(M)_{\infty}$. Theorem \ref{thm:main1} shows that if $\mB$ is degree-non-decreasing homogeneous basis, then $\mB$ is bi-ordering basis.

Now we construct a bi-ordering basis which is not degree non-decreasing
by using quasi-nilpotent fibration.
Let $B$ and $F$ be connected manifolds having the residually torsion-free nilpotent fundamental groups, and assume that $\pi_{2}(B)=1$. A fibration $F \rightarrow E \rightarrow B$ is called a {\it quasi-nilpotent fibration} if $\pi_{1}(B)$ acts on $H_{1}(F;\R)$ nilpotently. 

We can construct a bi-ordering basis of $R_{E} = \widehat{\R \pi_{1}(E)}$ as follows. First of all, let us take an arbitrary bi-ordering basis $\{v^{B}_{j}\} _{j \in \mathcal{J}}$ of $R_{B}= \widehat{\R \pi_{1}(B)}$. Since $\pi_{1}(B)$ acts on $H_{1}(F;\R)$ nilpotently, this fibration induces an exact sequence of the nilpotent completions of the fundamental groups \cite{b}.
 \[ 1 \rightarrow R_{F} \rightarrow R_{E} \rightarrow R_{B} \rightarrow 1. \]
 
Since $\pi_{1}(B)$ acts on $H_{1}(F;\R)$ nilpotently, there is a basis $\{X_{1},\ldots,X_{m}\}$ of $H_{1}(F;\R)$ such that $g(X_{i}) = X_{i} + \sum_{i'>i}^{m} a_{i'} X_{i'}$ holds for all $g\in \pi_{1}(B)$. Let us choose an ordered, degree non-decreasing basis $\{v^{F}_{i}\}_{i \in \mI}$ of $R_{F}=\widehat{\R \pi_{1}(F)}$ as the subset of the {\sf DegLex}-ordered monomials of $\R \dl X_{1},\ldots ,X_{m}\dr$. 

Now we are ready to give a bi-ordering basis of $R_{E}$. Let us take an ordered basis $\{ v^{B}_{j} v^{F}_{i}\}_{(j,i) \in\mathcal{J} \times  \mI}$ of $R_{E}$. Here the ordering of $\mathcal{J} \times  \mI$ is the lexicographical ordering defined by $(j,i) >(j',i')$ if and only if $j>j'$, or $j=j'$ and $i'>i$. 
From the choice of the basis $\{v^{F}_{i}\}$, $g(v^{F}_{i}) = v^{F}_{i} + \sum_{i'>i} c_{i'} v^{F}_{i'}$ holds for all $g \in \pi_{1}(B)$, and $v^{F}_{i} \cdot  v^{B}_{j} = v^{B}_{j}v^{F}_{i} + \sum_{i' >i} v^{B}_{j}v^{F}_{i'}$. This implies that the basis $\{v^{B}_{j}v^{F}_{i}\}_{(j,i)\in\mathcal{J} \times  \mI}$ is a bi-ordering basis.

Thus, we can construct a bi-ordering basis and bi-ordering using a sequence of quasi-nilpotent fibration. We call such a bi-ordering $<$ {\it generalized holonomy orderings}.
In algebraic point of view, this is a construction of bi-ordering of $\pi_{1}(E)$ from bi-orderings of $\pi_{1}(F)$ and $\pi_{1}(B)$ and is a special case of a group extension construction described in Lemma \ref{lem:extension}.

We close the paper by giving an example of known bi-orderings which are in fact a generalized holonomy ordering.

\begin{exam}[The fundamental groups of fiber-type arrangements \cite{kr}]
Recall that a hyperplane arrangement $\mA=\{H_{i}\}$ of $\C^{n}$ is called a {\it fiber-type arrangement} if its complement $M_{\mA}= \C^{n} -\bigcup_{H \in \mA} H$ has a tower of fibrations
\[ M_{\mA}=M_{n} \stackrel{p_{n}}{\rightarrow} M_{n-1} \stackrel{p_{n-1}}{\rightarrow} \cdots \stackrel{p_{2}}{\rightarrow} M_{1}= \C -\{0\} \]
where each fiber $F_{k}$ of $p_{k}$ is homeomorphic to $\C -\{ \textrm{finite points}\}$. See \cite{fr} for the precise definition. The braid arrangement is a typical example of fiber-type arrangement. It is known that for each fibration $p_{k}$, the action of $\pi_{1}(M_{k-1})$ on $H_{1}(F_{k})$ is trivial \cite{fr}. In particular, each fibration is quasi-nilpotent.
 
Thus, by using the tower of quasi-nilpotent fibrations we can construct a bi-ordering basis of $R_{M_{\mA}} = \widehat{\R \pi_{1}(M_{\mA})}$, and bi-ordering of $\pi_{1}(M_{\mA})$.

On the other hand, in \cite{kr} Kim-Rolfsen constructed a bi-ordering of $\pi_{1}(M_{\mA})$ by using the tower of quasi-nilpotent fibrations and the group extension construction (Lemma \ref{lem:extension}). Thus Kim-Rolfsen's bi-ordering can be seen as a generalized holonomy ordering.
\end{exam}

\end{document}